\documentclass[a4paper,reqno]{amsart}

\usepackage{amssymb}
\usepackage{latexsym}
\usepackage{amsmath}
\usepackage{euscript}
\usepackage{tikz}

      \def\dR{{\mathbb R}}

\def\cD{{\mathcal D}}      
      
      \def\cL{{\mathcal L}}
      \def\cO{{\mathcal O}}

\def\cal H{{\mathcal H}}

\def\R{\mathbb{R}}
\def\C{\mathbb{C}}

\def\N{\mathbb{N}}

\DeclareMathOperator{\dom}{dom}
\DeclareMathOperator{\ran}{ran}

\def\loc{{\text{\rm loc}}}

\def\spann{\textup{span}}
\def\phi{\varphi}
\def\d{\textup{d}}
\def\eps{\varepsilon}

\def\fa{\mathfrak{a}}

\DeclareMathOperator{\supp}{supp}

\DeclareMathOperator{\Imag}{Im}

\newtheorem{theorem}{Theorem}[section]
\newtheorem*{thm*}{Theorem}
\newtheorem{proposition}[theorem]{Proposition}
\newtheorem{corollary}[theorem]{Corollary}
\newtheorem{lemma}[theorem]{Lemma}

\theoremstyle{definition}
\newtheorem{definition}[theorem]{Definition}

\newtheorem*{ack}{Acknowledgement}

\numberwithin{equation}{section}

\title[Inverse problems with partial data on unbounded domains]{Inverse problems with partial data for elliptic operators on unbounded Lipschitz domains}

\author[J.~Behrndt]{Jussi Behrndt}

\address{Institut f\"ur Numerische Mathematik \\
Technische Universit\"at Graz \\
Steyrergasse 30\\
8010 Graz\\
Austria}
\email{behrndt@tugraz.at}

\author[J.~Rohleder]{Jonathan Rohleder}
\address{Stockholms universitet, Matematiska institutionen, 10691 Stockholm, Sweden}
\email{jonathan.rohleder@math.su.se}

\begin{document}

\begin{abstract}
For a second order formally symmetric elliptic differential expression we show that the knowledge of the 
Dirichlet-to-Neumann map or Robin-to-Dirichlet map for suitably many energies on an arbitrarily small open subset of the boundary determines 
the self-adjoint operator with a Dirichlet boundary condition or with a (possibly non-self-adjoint) Robin boundary condition uniquely up 
to unitary equivalence. These results hold for general Lipschitz domains, which can be unbounded and may have a non-compact boundary, 
and under weak regularity assumptions on the coefficients of the differential expression.
\end{abstract}

\maketitle

\section{Introduction}

Let $\cL$ be a uniformly elliptic, formally symmetric differential expression of the form
\begin{align}\label{eq:diffexpr}
 \cL = - \sum_{j,k = 1}^n \partial_j a_{jk} \partial_k + \sum_{j = 1}^n \big( a_j \partial_j - \partial_j \overline{ a_j} \big)  + a
\end{align}
on a possibly unbounded Lipschitz domain $\Omega$. For appropriate $\lambda \in \C$, the corresponding Dirichlet-to-Neumann map is given by
\begin{align*}
 M (\lambda) : H^{1/2} (\partial \Omega) \to H^{- 1/2} (\partial \Omega), \quad u_\lambda |_{\partial \Omega} \mapsto \partial_\cL u_\lambda |_{\partial \Omega},
\end{align*}
where $u_\lambda \in H^1 (\Omega)$ solves the differential equation $\cL u = \lambda u$, $u_\lambda |_{\partial \Omega}$ denotes the trace of $u_\lambda$ 
on the boundary $\partial \Omega$ and $\partial_\cL u_\lambda |_{\partial \Omega}$ is the conormal derivative of $u_\lambda$ on $\partial \Omega$ 
with respect to $\cL$. In the present paper it will be shown that the partial knowledge of $M (\lambda)$ on an arbitrarily small nonempty, relatively open subset $\omega$ of $\partial \Omega$ for a set of points $\lambda$ with an accumulation point determines the self-adjoint Dirichlet operator
\begin{align*}
 A_{\rm D} u = \cL u, \quad \dom A_{\rm D} = \left\{ u \in H^1 (\Omega) : \cL u \in L^2 (\Omega), u |_{\partial \Omega} = 0 \right\},
\end{align*}
and other realizations of $\cL$ with (possibly non-self-adjoint) Robin boundary conditions uniquely up to unitary equivalence in $L^2 (\Omega)$. We impose weak regularity assumptions on the coefficients, that is, $a_{jk}, a_j : \overline{\Omega} \to \C$ are bounded 
Lipschitz functions, $1 \leq j, k \leq n$, and $a : \Omega \to \R$ is measurable and bounded. We emphasize that $\Omega$ is an unbounded Lipschitz 
domain without any additional geometric restrictions, and that 
$\omega$ may be a bounded subset of $\partial \Omega$ even in the case that $\partial \Omega$ is unbounded.

The interplay between elliptic differential operators and their corresponding Dirichlet-to-Neumann maps is of particular interest for spectral theory and 
inverse problems, among them the famous Calder\'on problem, the multidimensional Gelfand inverse boundary spectral problem, and inverse scattering problems on Riemannian manifolds. 
In his famous paper~\cite{C80} A.~Calder\'on asked whether the uniformly positive coefficient $\gamma$ in the differential expression 
$- \nabla \cdot \gamma \nabla$ on a bounded domain $\Omega$ is uniquely determined by the Dirichlet-to-Neumann map on the boundary 
$\partial \Omega$ or on parts of the boundary; this corresponds to the case $a_{jk} = \gamma \delta_{jk}$, $a_j = a = 0$ in~\eqref{eq:diffexpr}, 
and $\gamma$ describes the isotropic conductivity of an inhomogeneous body. There is an extensive literature on this topic and 
uniqueness of the coefficient $\gamma$ from the knowledge of $M (0)$ has been shown under rather general regularity assumptions, see, e.g.,~\cite{AP06,N88,N96,NSU88,SU87} and
\cite{BU02,IUY10,KSU07,NS10} for results with partial data, as well as \cite{ALP05,DKSU09,LU89,SU03,S90,SU91} for the more general case of an 
anisotropic conductivity ($a_j = a = 0$ in~\eqref{eq:diffexpr}) and the surveys~\cite{U09,U14,U14-2}.
If $\Omega$ is an unbounded domain the situation is much more difficult since, very roughly speaking, the spectrum 
contains continuous parts. For conductivities that are constant outside compact sets, special unbounded domains (infinite slabs or transversally anisotropic geometries),
and magnetic Schr\"{o}dinger operators, uniqueness results were shown in \cite{CM16,CS15,DKLS16,I01,KKS15,K18,K19,KLU12,LTU03,LU01,LU10,P14,SW06}. 

In Gelfands inverse boundary spectral problem -- which is a variant of the inverse problems discussed in the present paper for bounded domains -- 
one reconstructs from the given boundary spectral data on a compact manifold (consisting of eigenvalues and boundary data of eigenfunctions of a self-adjoint elliptic operator) 
the manifold and its metric (up to gauge equivalence) with the help of the boundary control method; cf. \cite{AKKLT04,B88,B97,B07,BK92,KK98,KKL01,KL06} and \cite{KL97,L98} for the non-self-adjoint case. 
There is also a  strong recent interest in closely related problems in inverse scattering theory on compact and non-compact Riemannian manifolds; 
here the main theme is the reconstruction of 
the manifold and its Riemannian metric from the knowledge of the scattering matrix for the 
Laplace-Beltrami operator, see e.g. \cite{BK92,I04,IK14,IKL10,IKL14,IKL17,KKL01,LSU15}.

The inverse problems discussed in this paper are of a somewhat more abstract, but also more general nature. In Sections~\ref{333} and \ref{4444} it will be shown that the knowledge of the 
Dirichlet-to-Neumann map for a suitable set of points $\lambda$ with an accumulation point on an arbitrarily small open subset of the boundary 
determines the self-adjoint Dirichlet operator and other non-self-adjoint realizations
with mixed Dirichlet-Robin boundary conditions up to unitary equivalence. 
We treat here the general case of an unbounded Lipschitz domain without any additional geometric restrictions and assume 
weak regularity assumptions on the coefficients of the elliptic differential expression.
We emphasize 
that unitary equivalence determines the spectral properties,
so that, in particular, 
the isolated and embedded eigenvalues, continuous, essential, absolutely continuous and singular continuous spectra are uniquely determined 
by the partial knowledge of the Dirichlet-to-Neumann map. Finally, in Section~\ref{55555} another variant 
of our uniqueness result is provided for self-adjoint Robin realizations, where instead of the Dirichlet-to-Neumann map a Robin-to-Dirichlet map on an open subset
of the boundary is considered. 
The main results in this paper complement earlier results for bounded domains from \cite{BR12}, see also~\cite{O17}, where the uniqueness problem is substantially easier since
all spectral singularities are discrete eigenvalues, and hence poles of the Dirichlet-to-Neumann map. 
Our proofs in the present paper are based on more elaborate methods from the extension theory of symmetric operators and the
spectral theory of elliptic operators; related techniques were also developed and used in \cite{BR15,BR16} for the spectral analysis of 
Schr\"{o}dinger and more general elliptic operators.
In this context we also refer the reader to~\cite{AGW14,GM09,GM11,G08,G11,G11b,G12,M10,MPS16,P08,PR09,P16} for some recent related papers on spectral theory 
of elliptic differential operators, to the classical contributions \cite{G68,V52}, and to~\cite{AE11,AE15,AEKS14,BGW09,BMNW08,BMNW17,GM08} for operator-theoretic 
approaches to Dirichlet-to-Neumann and Robin-to-Dirichlet maps.

\section{Preliminaries}\label{sec:2}

In this section we provide some preliminaries on elliptic differential operators on possibly unbounded Lipschitz domains. Throughout this paper we assume that 
$\Omega \subset \R^n$, $n \geq 2$, is a connected Lipschitz domain in the sense of, e.g.,~\cite[VI.3]{S70}, that is, $\Omega$ is an open, 
connected set with a nonempty boundary $\partial \Omega$ and there exist $\eps > 0$, $N \in \N$, $M > 0$ and (finitely or infinitely many) open sets $U_1, U_2, \dots$ with the following properties.
\begin{enumerate}
 \item For each $x \in \partial \Omega$ there exists $j$ such that the open ball $B (x, \eps)$ of radius $\eps$ centered at $x$ is contained in $U_j$.
 \item No point of $\R^n$ is contained in more than $N$ of the $U_j$.
 \item For each $j$ there exists a function $\zeta_j : \R^{n - 1} \to \R$ with
 \begin{align*}
  |\zeta_j (x) - \zeta_j (y)| \leq M |x - y|, \quad x, y \in \R^{n - 1},
 \end{align*}
 such that (up to a possible rotation of coordinates) the Lipschitz hypographs
 \begin{align*}
 \Omega_j := \left\{ (x_1, \dots, x_n)^\top \in \R^n : x_n < \zeta_j (x_1, \dots, x_{n-1}) \right\}
 \end{align*}
 satisfy $U_j \cap \Omega = U_j \cap \Omega_j$.
\end{enumerate}

We are particularly interested in the case that $\Omega$ is unbounded. Note that the boundary $\partial\Omega$ may be noncompact.
It can be described by the graphs of countably many Lipschitz functions with a joint Lipschitz constant.

In the following we denote by $H^s (\Omega)$ and $H^t (\partial \Omega)$ the Sobolev spaces of order $s \in \R$ on $\Omega$ and of order $t\in [-1,1]$ on 
its boundary $\partial \Omega$, respectively. We point out that under the above assumptions on $\Omega$ many typical properties of Sobolev spaces on bounded Lipschitz 
domains and their boundaries remain true. For instance, by the same proofs as provided in~\cite[Theorem~3.37 and Theorem~3.40]{M00} for bounded domains, one verifies that there exists a continuous, surjective trace operator from $H^1 (\Omega)$ onto $H^{1/2} (\partial \Omega)$ and that its kernel coincides with $H_0^1 (\Omega)$, the closure of $C_0^\infty (\Omega)$ in $H^1 (\Omega)$. In the following we denote the trace of a function $u \in H^1 (\Omega)$ by~$u |_{\partial \Omega}$.

On $\Omega$ let us consider the differential expression $\cL$ in~\eqref{eq:diffexpr} satisfying the uniform ellipticity condition
\begin{align}\label{eq:elliptic}
 \sum_{j,k=1}^n a_{jk}(x) \xi_j\xi_k\geq E \sum_{k=1}^n\xi_k^2, \quad \xi=(\xi_1,\dots\xi_n)^\top \in \dR^n,\,\, x \in \overline \Omega,
\end{align}
for some $E > 0$. We assume that 
\begin{align}\label{eq:ajkaj}
 a_{jk}, a_j : \overline{\Omega} \to \C~\text{are bounded Lipschitz functions}, \quad 1 \leq j, k \leq n, 
\end{align}
\begin{align}\label{eq:ajk}
 a_{jk} (x) = \overline{a_{kj} (x)}, \quad x \in \overline{\Omega},
\end{align}
and that 
\begin{align}\label{eq:a}
 a : \Omega \to \R~\text{is measurable and bounded}.
\end{align}

In the following we make use of the conormal derivative (with respect to $\cL$). For a function $u \in H^1 (\Omega)$ such that $\cL u \in L^2 (\Omega)$ 
in the sense of distributions, the conormal derivative of $u$ at $\partial \Omega$ with respect to $\cL$ is defined as the unique 
$\psi \in H^{- 1/2} (\partial \Omega)$ which satisfies the identity
\begin{align*}
 \fa [u, v] = (\cL u, v)_{L^2(\Omega)} + (\psi, v |_{\partial \Omega})_{\partial \Omega}
\end{align*}
for all $v \in H^1 (\Omega)$, where $(\cdot, \cdot)_{L^2(\Omega)}$ is the inner product in $L^2 (\Omega)$, $(\cdot, \cdot)_{\partial \Omega}$ 
denotes the (sesquilinear) duality of $H^{- 1/2} (\partial \Omega)$ and $H^{1/2} (\partial \Omega)$, and 
\begin{align}\label{eq:aForm}
  \fa [u, v] & = \int_\Omega \bigg( \sum_{j,k = 1}^n a_{jk} \partial_k u \cdot \overline{\partial_j v} + \sum_{j = 1}^n \big( a_j \partial_j u \cdot  \overline v + \overline{a_j} u \cdot \overline{\partial_j v} \big)  + a u \overline v \bigg) \d x;
\end{align}
cf.~\cite[Lemma~4.3]{M00}. We shall use the notation $\psi = \partial_\cL u |_{\partial \Omega}$.

\section{An inverse problem for the Dirichlet operator with partial Dirichlet-to-Neumann data}\label{333}

In this section we prove that the partial knowledge of the Dirichlet-to-Neumann map determines the Dirichlet
realization of $\cL$ in $L^2(\Omega)$ uniquely up to unitary equivalence. Recall first that \eqref{eq:ajkaj}--\eqref{eq:a} ensure that the Dirichlet operator 
\begin{align}\label{eq:AD}
 A_{\rm D} u = \cL u, \quad \dom A_{\rm D} = \left\{ u \in H^1 (\Omega) : \cL u \in L^2 (\Omega), u |_{\partial \Omega} = 0 \right\},
\end{align}
is a semibounded self-adjoint operator in $L^2(\Omega)$ since it corresponds to the closed semibounded sesquilinear form 
 $$\fa_{\rm D}[u,v]:=\fa[u,v],\qquad u,v\in\dom\fa_{\rm D}=H^1_0(\Omega),$$
via the first representation theorem; cf. \cite[Theorem~VI.2.1]{Kato} and~\cite[Chapter~VI]{EE87}.

In order to define the Dirichlet-to-Neumann map associated with $\cL$ on the boundary of the unbounded Lipschitz domain $\Omega$  we need the following lemma,
which is well known for bounded domains and remains valid in the  unbounded case. 
For the convenience of the reader we provide a short proof. By $\rho (A_{\rm D})$ we denote the resolvent set of $A_{\rm D}$, i.e., the complement of the spectrum.

\begin{lemma}\label{lem:BVPwelldef}
For each $\lambda \in \rho (A_{\rm D})$ and each $\phi \in H^{1/2} (\partial \Omega)$ the boundary value problem
\begin{align}\label{eq:BVP}
 \cL u = \lambda u, \qquad u |_{\partial \Omega} = \phi,
\end{align}
has a unique solution $u_\lambda \in H^1 (\Omega)$.
\end{lemma}

\begin{proof}
Let $\lambda \in \rho (A_{\rm D})$ and $\phi \in H^{1/2} (\partial \Omega)$. Since the trace map is surjective from $H^1 (\Omega)$ to $H^{1/2} (\partial \Omega)$ 
there exists (a non-unique) $w \in H^1 (\Omega)$ with $w |_{\partial \Omega} = \phi$. Let $\fa$ be the symmetric sesquilinear form on $H^1 (\Omega)$ defined in~\eqref{eq:aForm}. 
It follows from~\eqref{eq:ajkaj} and~\eqref{eq:a} that there exists $C > 0$ such that
\begin{align}\label{eq:C}
 |\fa [u, v]| \leq C \| u \|_{H^1 (\Omega)} \| v \|_{H^1 (\Omega)}, \quad u, v \in H^1 (\Omega),
\end{align}
where $\| \cdot \|_{H^1 (\Omega)}$ denotes the norm in $H^1 (\Omega)$. In particular, the antilinear mapping 
\begin{align*}
 F_{w, \zeta} : H_0^1 (\Omega) \to \C, \quad v \mapsto  \fa [w, v] + \zeta (w, v)_{L^2 (\Omega)},
\end{align*}
is bounded on $H_0^1 (\Omega)$ for each $\zeta \in \R$; hence $F_{w, \zeta}$ belongs to the antidual of $H_0^1 (\Omega)$. Moreover, it follows from~\eqref{eq:C} and the ellipticity condition~\eqref{eq:elliptic} that we can fix $\zeta_0 \in \R$ such that
\begin{align}\label{eq:innerprod}
 \fa [u, v] + \zeta_0 (u, v)_{L^2(\Omega)}, \quad u, v \in H_0^1 (\Omega),
\end{align}
defines an inner product on $H_0^1 (\Omega)$ with an induced norm that is equivalent to the norm $\| \cdot \|_{H^1 (\Omega)}$. In particular, $H_0^1 (\Omega)$ 
equipped with the inner product in~\eqref{eq:innerprod} is a Hilbert space. By the Fr\'echet--Riesz theorem there exists a unique $u_0 \in H_0^1 (\Omega)$ 
such that
\begin{align*}
 \fa [u_0, v] + \zeta_0 (u_0, v)_{L^2 (\Omega)} = F_{w, \zeta_0} (v) =   \fa [w, v] + \zeta_0 (w, v)_{L^2 (\Omega)}, \quad v \in H_0^1 (\Omega).
\end{align*}
Consequently, $\fa [u_0 - w, v] + \zeta_0 (u_0-w,v)_{L^2(\Omega)}=0$ 
for all $v \in H_0^1 (\Omega)$, which implies $\cL (u_0 - w) +\zeta_0 (u_0-w)= 0$ in the distributional sense. 
For $\lambda\in\rho(A_{\rm D})$ it follows,
in particular, that $(\cL - \lambda) (u_0 - w) \in L^2 (\Omega)$. Let us set
\begin{align*}
 u_\lambda = u_0 - w - (A_{\rm D} - \lambda)^{-1} (\cL - \lambda) (u_0 - w) \in H^1 (\Omega).
\end{align*}
Then $u_\lambda |_{\partial \Omega} = w |_{\partial \Omega} = \phi$ and $(\cL - \lambda) u_\lambda = 0$. Thus $u_\lambda$ is a solution of~\eqref{eq:BVP}.

In order to prove uniqueness let $v_\lambda \in H^1 (\Omega)$ be a further solution of~\eqref{eq:BVP}. Then we have
\begin{align*}
 \cL (u_\lambda - v_\lambda) = \lambda (u_\lambda - v_\lambda) \quad \text{and} \quad (u_\lambda - v_\lambda) |_{\partial \Omega} = 0,
\end{align*}
that is, $(u_\lambda - v_\lambda) \in \ker (A_{\rm D} - \lambda)$. Since $\lambda \in \rho (A_{\rm D})$, it follows $u_\lambda = v_\lambda$.
\end{proof}

Lemma~\ref{lem:BVPwelldef} ensures that the Dirichlet-to-Neumann map in the following definition is well-defined.

\begin{definition}
For $\lambda \in \rho (A_{\rm D})$ the {\em Dirichlet-to-Neumann map} $M (\lambda)$ is defined  by
\begin{align*}
M (\lambda) : H^{1/2} (\partial \Omega) \to H^{- 1/2} (\partial \Omega),\qquad M (\lambda) u_\lambda |_{\partial \Omega} := \partial_\cL u_\lambda |_{\partial \Omega},
\end{align*}
for each $u_\lambda \in H^1 (\Omega)$ satisfying $\cL u_\lambda = \lambda u_\lambda$.
\end{definition}

For $\lambda \in \rho (A_{\rm D})$ we will also make use of the {\em Poisson operator} $\gamma (\lambda)$ defined by 
\begin{align}\label{eq:Poisson}
 \gamma(\lambda): H^{1/2} (\partial \Omega) \to L^2 (\Omega), \qquad \gamma (\lambda) u_\lambda |_{\partial \Omega} := u_\lambda,
\end{align}
for any $u_\lambda \in H^1 (\Omega)$ such that $\cL u_\lambda = \lambda u_\lambda$; cf.~Lemma~\ref{lem:BVPwelldef}.

We collect some properties of the Dirichlet-to-Neumann map and the Poisson operator in the following lemma. 
Its proof is analogous to the case of a bounded Lipschitz domain carried out in~\cite[Lemma~2.4]{BR12}.

\begin{lemma}\label{lem:DNproperties}
For $\lambda, \mu \in \rho (A_{\rm D})$ let $\gamma (\lambda), \gamma (\mu)$ be the Poisson operators and let $M (\lambda), M (\mu)$ be the Dirichlet-to-Neumann maps. Then the following assertions hold.
\begin{enumerate}
 \item $\gamma (\lambda)$ is bounded and its adjoint $\gamma (\lambda)^* : L^2 (\Omega) \to H^{- 1/2} (\partial \Omega)$ is given by
\begin{align*}
 \gamma (\lambda)^* u = - \partial_\cL \big( (A_{\rm D} - \overline{\lambda})^{-1} u \big) |_{\partial \Omega}, \quad u \in L^2 (\Omega).
\end{align*}
 \item The identity
 \begin{align*}
  \gamma (\lambda) = \big(I + (\lambda - \mu) (A_{\rm D} - \lambda)^{-1} \big) \gamma (\mu)
 \end{align*}
 holds.
 \item $M (\lambda)$ is a bounded operator from $H^{1/2} (\partial \Omega)$ to $H^{- 1/2} (\partial \Omega)$, 
 the operator function $\lambda \mapsto M (\lambda)$ is holomorphic on $\rho (A_{\rm D})$, and
 \begin{align*}
  (\Imag\mu)\|\gamma (\mu) \phi\|_{L^2 (\Omega)}^2 = - \Imag ( M (\mu) \phi, \phi )_{\partial \Omega}
 \end{align*}
 holds for all $\phi \in H^{1/2} (\partial \Omega)$.
\end{enumerate}
\end{lemma}

The next theorem is the main result in this section; one can view it as a generalized variant of the multidimensional Gelfand inverse boundary spectral problem 
with partial data on arbitrary unbounded Lipschitz domains. Instead of determining coefficients up to gauge equivalence here an operator uniqueness result is obtained. 
Roughly speaking Theorem~\ref{thm:main} states that the knowledge of the Dirichlet-to-Neumann map $M(\lambda)$ on a
nonempty open subset $\omega$ of the boundary $\partial\Omega$ for sufficiently many $\lambda$ determines the Dirichlet operator uniquely up to unitary equivalence. 
For bounded Lipschitz domains such a result was shown in \cite{BR12}, see also~\cite{O17}.

\begin{theorem}\label{thm:main}
Let $\cL_1, \cL_2$ be two uniformly elliptic differential expressions on $\Omega$ of the form~\eqref{eq:diffexpr} with coefficients 
$a_{jk, 1}, a_{j, 1}, a_1$ and $a_{jk, 2}, a_{j, 2}, a_2$, respectively, satisfying \eqref{eq:ajkaj}--\eqref{eq:a}. 
Denote by $A_{\rm D, 1}$, $A_{\rm D, 2}$ and $M_1 (\lambda), M_2 (\lambda)$ 
the corresponding self-adjoint Dirichlet operators and Dirichlet-to-Neumann maps, respectively. Assume that $\omega \subset \partial \Omega$ is an open, 
nonempty set such that
\begin{align}\label{eq:Mgleich}
 ( M_1 (\lambda) \phi, \phi)_{\partial \Omega} = ( M_2 (\lambda) \phi, \phi)_{\partial \Omega}, 
 \quad \phi\in H^{1/2} (\partial \Omega),\, \supp \phi \subset \omega,
\end{align}
holds for all $\lambda\in\cD$,
where $\cD \subset \rho (A_{\rm D, 1}) \cap \rho (A_{\rm D, 2})$ is a set with an accumulation point in $\rho (A_{\rm D, 1}) \cap \rho (A_{\rm D, 2})$. 
Then there exists a unitary operator $U$ in $L^2 (\Omega)$ such that
\begin{equation}\label{juhu}
 A_{\rm D, 2} = U A_{\rm D, 1} U^*
\end{equation}
holds.
\end{theorem}

Before we provide a proof of the theorem, let us point out that unitary equivalence of self-adjoint operators implies that their spectra coincide.

\begin{corollary}
Let the assumptions be as in Theorem~\ref{thm:main}. Then $\mu \in \R$ belongs to the point (discrete, essential,
continuous, absolutely continuous, singular continuous) spectrum of $A_{\rm D, 1}$ if and only if 
$\mu$ belongs to the point (discrete, essential,
continuous, absolutely continuous, singular continuous) spectrum of $A_{\rm D, 2}$, respectively.
\end{corollary}

\begin{proof}[Proof of Theorem~\ref{thm:main}]
The proof will be carried out in two steps. In the first step an isometric operator defined on a subspace of $L^2 (\Omega)$ is constructed; this step follows the strategy of the proof of~\cite[Theorem 1.3]{BR12} but is given here for completeness. 
In the second step we show that this operator extends to a unitary operator such that \eqref{juhu} holds.

{\bf Step~1.} Let $\cL_1$, $\cL_2$ be differential expressions as in the theorem and let $A_{\rm D, 1}, A_{\rm D, 2}$ and $M_1 (\lambda)$, $M_2 (\lambda)$ be the 
corresponding Dirichlet operators and Dirichlet-to-Neumann maps, respectively. Moreover, denote by $\gamma_1 (\lambda)$ and $\gamma_2 (\lambda)$ the corresponding 
Poisson operators as in~\eqref{eq:Poisson}. Assume that~\eqref{eq:Mgleich} holds for all $\lambda \in \cD$. Since 
$( M_i (\cdot) \phi, \phi)_{\partial \Omega}$ is 
holomorphic on $\rho (A_{{\rm D}, i})$ for all $\phi \in H^{1/2}(\partial\Omega)$ with $\supp\phi \subset \omega$, 
$i = 1, 2$, and $\cD$ has an accumulation point in $\rho (A_{\rm D, 1}) \cap \rho (A_{\rm D, 2})$, it follows that
\begin{align*}
 ( M_1 (\lambda) \phi, \phi)_{\partial \Omega} = ( M_2 (\lambda) \phi, \phi)_{\partial \Omega}, 
 \quad \phi \in H^{1/2} (\partial \Omega),\, \supp \phi \subset \omega,
\end{align*}
holds for all $\lambda\in\rho (A_{\rm D, 1}) \cap \rho (A_{\rm D, 2})$.
With Lemma~\ref{lem:DNproperties}~(iii) for all $\mu \in \C \setminus \R$ and all 
$\phi \in H^{1/2}(\partial\Omega)$ with $\supp\phi\subset\omega$ we obtain
\begin{equation}\label{eq:isometry}
\begin{split}
 \|\gamma_1(\mu) \phi\|_{L^2(\Omega)}^2 & = - \frac{\Imag ( M_1(\mu) \phi,\phi)_{\partial \Omega}}{\Imag \mu} \\
 & = - \frac{\Imag ( M_2 (\mu) \phi,\phi)_{\partial \Omega}}{\Imag \mu} = \|\gamma_2 (\mu) \phi\|_{L^2(\Omega)}^2.
\end{split}
\end{equation}
Let us define a linear mapping $V$ in $L^2(\Omega)$ on the domain 
\begin{equation}\label{domvii}
 \dom V=\spann \big\{\gamma_1(\mu)\phi : \phi\in H^{1/2}(\partial\Omega),\,\supp\phi\subset\omega,\,\mu \in \C \setminus \R \big\}
\end{equation}
by setting
\begin{align}\label{eq:V}
 V \gamma_1 (\mu) \phi = \gamma_2 (\mu) \phi, \quad \phi\in H^{1/2}(\partial\Omega),\,\supp\phi\subset\omega,\,\mu \in \C \setminus \R,
\end{align}
and extending it by linearity to all of $\dom V$. It follows from~\eqref{eq:isometry} that $V$ is a well-defined, isometric operator in $L^2(\Omega)$ with
\begin{align*}
 \ran V=\spann \big\{\gamma_2(\mu)\phi :\phi\in H^{1/2}(\partial\Omega),\,\supp\phi\subset\omega,\, \mu \in \C \setminus \R \big\}.
\end{align*}
Moreover, if we fix $\lambda \in \C \setminus \R$ then by Lemma~\ref{lem:DNproperties}~(ii) we have $\ran (A_{\rm D, 1} - \lambda)^{-1} \gamma_1 (\mu) \subset \dom V$ and
\begin{align*}
 V (A_{\rm D, 1} - \lambda)^{-1} \gamma_1(\mu)\phi & = V \frac{\gamma_1(\lambda) \phi- \gamma_1(\mu)\phi}{\lambda - \mu} =  \frac{\gamma_2(\lambda) \phi- \gamma_2(\mu)\phi}{\lambda - \mu}\\
  &=  (A_{\rm D, 2} - \lambda)^{-1}  \gamma_2(\mu)\phi =  (A_{\rm D, 2} - \lambda)^{-1}  V\gamma_1(\mu)\phi
\end{align*}
for all $\mu \in \C \setminus \R$ with $\mu \neq \lambda$ and all $\phi \in  H^{1/2}(\partial\Omega)$ with
$\supp\phi\subset\omega$. By linearity this implies
\begin{align}\label{eq:Vidjj}
 V (A_{\rm D, 1} - \lambda)^{-1} \upharpoonright H_\lambda = (A_{\rm D, 2} - \lambda)^{-1} V \upharpoonright H_\lambda,
\end{align}
where $H_\lambda$ is the subspace of $\dom V$ given by
\begin{align}\label{eq:Hlambda}
 H_\lambda = \spann \big\{\gamma_1(\mu)\phi : \phi\in H^{1/2}(\partial\Omega),\,\supp\phi\subset\omega,\,\mu \in \C \setminus \R, \mu \neq \lambda \big\}.
\end{align}

{\bf Step~2.} Let us show that the linear space $\dom V$ in \eqref{domvii} 
is dense in $L^2 (\Omega)$. For this choose a Lipschitz domain $\widetilde \Omega$ 
such that $\Omega\subset\widetilde \Omega$, $\partial \Omega \setminus \omega \subset \partial \widetilde \Omega$, 
and $\widetilde \Omega \setminus \Omega$ contains an open ball $\cO$, and such that $\cL_1$ admits a uniformly elliptic, formally symmetric 
extension $\widetilde\cL_1$ to $\widetilde\Omega$ with coefficients satisfying \eqref{eq:ajkaj}--\eqref{eq:a} on $\widetilde \Omega$. 
Let $\widetilde A_{\rm D, 1}$ 
denote the self-adjoint Dirichlet operator associated with $\widetilde \cL_1$ in $L^2 (\widetilde \Omega)$,
\begin{align*}
 \widetilde A_{\rm D, 1} \widetilde u = \widetilde \cL_1 \widetilde u, \quad \dom \widetilde A_{\rm D, 1} = 
 \bigl\{ \widetilde u \in H^1 (\widetilde \Omega) : \widetilde \cL_1 \widetilde u \in L^2 (\widetilde \Omega), \widetilde u |_{\partial \widetilde \Omega} = 0 \bigr\}.
\end{align*}
Since $\widetilde A_{\rm D, 1}$ is semibounded from below, we can assume without loss of generality that this operator has a positive lower bound $\eta$. 
In fact, when a constant is added to the zero order term of $\cL_1$ (and $\widetilde\cL_1$) the linear space $\dom V$ in \eqref{domvii} remains the same.

For each $\widetilde v \in L^2 (\widetilde \Omega)$ such that $\widetilde v$ vanishes on $\Omega$ we define
\begin{align*}
 \widetilde u_{\mu, \widetilde v} = (\widetilde A_{\rm D, 1} - \mu)^{-1} \widetilde v, \quad \mu \in \C \setminus \R. 
\end{align*}
Moreover, denote by $u_{\mu, \widetilde v}$ the restriction of $\widetilde u_{\mu, \widetilde v}$ to $\Omega$. 
Then $u_{\mu, \widetilde v} \in H^1 (\Omega)$, $\cL_1 u_{\mu, \widetilde v} = \mu u_{\mu, \widetilde v}$, and 
$\supp (u_{\mu, \widetilde v} |_{\partial \Omega}) \subset \omega$, that is, with 
$\phi := u_{\mu, \widetilde v} |_{\partial \Omega} \in H^{1/2}(\partial\Omega)$ we have $u_{\mu, \widetilde v} = \gamma (\mu) \phi$ and $\supp\phi\subset\omega$; in particular, $u_{\mu, \widetilde v} \in \dom V$ holds for all $\mu \in \C \setminus \R$ and all $\widetilde v \in L^2 (\widetilde \Omega)$ with $\widetilde v |_\Omega = 0$.

Let $u \in L^2 (\Omega)$ such that $u$ is orthogonal to $\dom V$. Then the extension $\widetilde u$ of $u$ by zero to $\widetilde \Omega$ satisfies
\begin{align*}
 0 = (u, u_{\overline \mu, \widetilde v})_{L^2 (\Omega)} 
 = \bigl( \widetilde u, (\widetilde A_{\rm D, 1} - \overline \mu)^{-1} \widetilde v \bigr)_{L^2 (\widetilde \Omega)} 
 = \bigl( (\widetilde A_{\rm D, 1} - \mu)^{-1} \widetilde u, \widetilde v \bigr)_{L^2 (\widetilde \Omega)} 
\end{align*}
for all $\mu \in \C \setminus \R$ and all $\widetilde v \in L^2 (\widetilde \Omega)$ with $\widetilde v |_\Omega = 0$. Hence
\begin{align}\label{vanish}
  \bigl( (\widetilde A_{\rm D, 1} - \mu)^{-1} \widetilde u \bigr) \big|_{\widetilde \Omega \setminus \Omega} = 0, \quad \mu \in \C \setminus \R.
\end{align}
Following an idea from~\cite[Section~3]{BS01} we define the operator semigroup 
\begin{align*}
 T (t) = e^{- t \sqrt{\widetilde A_{\rm D, 1}}}, \quad t \geq 0,
\end{align*}
generated by the square root of $\widetilde A_{\rm D, 1}$. Then $t \mapsto T (t) \widetilde u$ is twice differentiable with
\begin{align*}
 \partial_t^2 T (t) \widetilde u = \widetilde A_{\rm D, 1} T (t) \widetilde u, \quad t > 0,
\end{align*}
from which we conclude
\begin{align}\label{Lkern}
 \big(- \partial_t^2 + \widetilde \cL_1 \big) T (t) \widetilde u = 0, \quad x \in \widetilde \Omega,\, t > 0,
\end{align}
in the distributional sense. Note that
$$(x, t) \mapsto \bigl(e^{- t \sqrt{\widetilde A_{\rm D, 1}}} \widetilde u\bigr) (x)\in L^2 (\widetilde \Omega \times (0, \infty) ).$$  Since the differential expression $\cL_1$ 
is uniformly elliptic on $\widetilde \Omega$, regularity theory implies 
$e^{- t \sqrt{\widetilde A_{\rm D, 1}}} \widetilde u \in H^2_{\loc} (\widetilde \Omega \times (0, \infty) )$.
For any real numbers $a, b$, $a < b$, which are no eigenvalues of $\widetilde A_{\rm D, 1}$ the Stone formula
\begin{align*}
 E_1 ((a, b)) \widetilde u = \lim_{\eps \searrow 0} \frac{1}{2 \pi i} \left( \int_a^b \bigl(\widetilde A_{\rm D, 1} - (z + i \eps) \bigr)^{-1} - 
        \bigl(\widetilde A_{\rm D, 1} - (z - i \eps) \bigr)^{-1} \d z \right) \widetilde u
\end{align*}
for the spectral measure $E_1 (\cdot)$ of $\widetilde A_{\rm D, 1}$ and~\eqref{vanish} imply 
$ (E_1 ((a, b)) \widetilde u) |_{\widetilde \Omega \setminus \Omega} = 0$. Thus, in particular, for each $t \geq 0$ 
\begin{align}\label{vanish2}
 \left( e^{- t \sqrt{\widetilde A_{\rm D, 1}}} \widetilde u \right) \Big|_{\widetilde \Omega \setminus \Omega} = \left( \int_\eta^\infty e^{- t \sqrt{z}} \d E_1 (z) \widetilde u \right) \Big|_{\widetilde \Omega \setminus \Omega} = 0.
\end{align}
By~\eqref{vanish2}, 
$e^{- t \sqrt{\widetilde A_{\rm D, 1}}} \widetilde u$ vanishes on the nonempty, open set $\cO \times (0, \infty)$, and~\eqref{Lkern} and unique continuation yield $T (t) \widetilde u = 0$ identically on $\widetilde \Omega$ for all $t > 0$, see, e.g.,~\cite{W93}. Thus, taking the limit $t \searrow 0$ we obtain $\widetilde u = 0$ and, hence, $u = 0$. Thus $\dom V$ is dense in $L^2 (\Omega)$. Analogously one shows that $\ran V$ is dense in $L^2 (\Omega)$.

To summarize, the operator $V$ in~\eqref{eq:V} is densely defined and isometric in $L^2 (\Omega)$ with a dense range. Hence it extends by continuity to a unitary 
operator $U : L^2 (\Omega) \to L^2 (\Omega)$. Moreover, note that the space $H_\lambda \subset \dom V$ in~\eqref{eq:Hlambda} 
is dense in $L^2 (\Omega)$ as well since $(\gamma (\mu) \phi, u)_{L^2 (\Omega)} = 0$ for all $\mu \in \C \setminus \R$ with $\mu \neq \lambda$ 
and all $\phi \in H^{1/2} (\partial \Omega)$ with $\supp \phi \subset \omega$ implies, by continuity, 
$(\gamma (\mu) \phi, u)_{L^2 (\Omega)} = 0$ for all $\mu \in \C \setminus \R$ and all $\phi \in H^{1/2} (\partial \Omega)$ with $\supp \phi \subset \omega$ and hence $u = 0$. Therefore the identity~\eqref{eq:Vidjj} extends to
\begin{align*}
 U (A_{\rm D, 1} - \lambda)^{-1} = (A_{\rm D, 2} - \lambda)^{-1} U,
\end{align*}
which implies $U \dom A_{\rm D, 1} = \dom A_{\rm D, 2}$ and $A_{\rm D, 2} = U A_{\rm D, 1} U^*$. This completes the proof of Theorem~\ref{thm:main}.
\end{proof}

\section{An inverse problem for a mixed non-self-adjoint Dirichlet--Robin operator with partial Dirichlet-to-Neumann data}\label{4444}

In this section 
we consider non-self-adjoint operators with mixed Dirichlet--Robin boundary conditions. 
We shall provide a variant of Theorem~\ref{thm:main} for $m$-sectorial elliptic operators satisfying a Robin boundary condition on an open subset $\omega\subset\partial\Omega$
and Dirichlet boundary conditions on $\partial\Omega\setminus\omega$. Here the knowledge of the Dirichlet-to-Neumann map is assumed locally at the 
same subset $\omega$ of $\partial\Omega$ on which the Robin condition is given.

In order to define the operators under consideration, let us set
\begin{align*}
 H^{1/2}_\omega = \overline{\big\{ \phi \in H^{1/2} (\partial \Omega) : \supp \phi \subset \omega \big\}},
\end{align*}
where the closure is taken in $H^{1/2} (\partial \Omega)$. Let $\theta\in L^\infty(\partial\Omega)$ be a complex-valued function such that $\theta\vert_{\partial\Omega\setminus\omega}=0$,
and consider the quadratic form
\begin{align*}
 \fa_{\theta, \omega} [u, v] = \fa [u, v] + ( \theta u |_{\partial \Omega}, v |_{\partial \Omega})_{\partial \Omega}, 
 \quad \dom \fa_{\theta, \omega} = \left\{ u \in H^1 (\Omega) : u |_{\partial \Omega} \in H^{1/2}_\omega \right\},
\end{align*}
where $\fa$ is given in~\eqref{eq:aForm}. One verifies that $\fa_{\theta, \omega}$ is a densely defined, sectorial, closed form in $L^2(\Omega)$ and gives rise to the $m$-sectorial operator 
\begin{align}\label{eq:Athetaomega}
\begin{split}
  A_{\theta, \omega} u & = \cL u, \\
 \dom A_{\theta, \omega} & = \big\{ u \in H^1 (\Omega) : \cL u \in L^2 (\Omega), \partial_\cL u |_{\omega} + \theta u |_{\omega} = 0, u |_{\partial \Omega} \in H^{1/2}_\omega \big\};
\end{split}
\end{align}
this operator realization of $\cL$ in $L^2(\Omega)$ is subject to a Dirichlet boundary condition on $\partial \Omega \setminus \omega$
and the Robin boundary condition $\partial_\cL u |_{\omega} + \theta u |_{\omega} = 0$ on $\omega$, which is understood as  
\begin{equation}\label{bcjussi}
 \bigl( \partial_\cL u |_{\partial\Omega} + \theta u |_{\partial\Omega},\phi\bigr)_{\partial\Omega}=0,\qquad \phi\in H^{1/2} (\partial \Omega), \supp \phi \subset \omega.
\end{equation}
Note also that for a real-valued $\theta\in L^\infty(\partial\Omega)$ such that $\theta\vert_{\partial\Omega\setminus\omega}=0$ 
the operator $A_{\theta, \omega}$ in \eqref{eq:Athetaomega} is self-adjoint in $L^2(\Omega)$ 
and semibounded from below.

\begin{theorem}\label{thm:main-j}
Let $\cL_1, \cL_2$ be two uniformly elliptic differential expressions on $\Omega$ of the form~\eqref{eq:diffexpr} 
with coefficients $a_{jk, 1}, a_{j, 1}, a_1$ and $a_{jk, 2}, a_{j, 2}, a_2$, respectively, satisfying \eqref{eq:ajkaj}--\eqref{eq:a}, and let
$M_1 (\lambda), M_2 (\lambda)$ be the corresponding Dirichlet-to-Neumann maps. Assume that $\omega \subset \partial \Omega$ is an open, nonempty set such that
\begin{align}\label{eq:Mgleich22}
 ( M_1 (\lambda) \phi, \phi)_{\partial \Omega} = ( M_2 (\lambda) \phi, \phi)_{\partial \Omega}, \qquad \phi \in H^{1/2} (\partial \Omega), 
 \supp \phi \subset \omega,
\end{align}
holds for all $\lambda \in \cD$, where $\cD \subset \rho (A_{\rm D, 1}) \cap \rho (A_{\rm D, 2})$ is a set with an accumulation point in $\rho (A_{\rm D, 1}) \cap \rho (A_{\rm D, 2})$. 
Let $\theta\in L^\infty(\partial\Omega)$ be a complex-valued function such that $\theta\vert_{\partial\Omega\setminus\omega}=0$ and denote by $A_{\theta, \omega, 1}$ and $A_{\theta, \omega, 2}$
the $m$-sectorial operators associated with $\cL_1$ and $\cL_2$, respectively, as in~\eqref{eq:Athetaomega}.
Then there exists a unitary operator $U$ in $L^2 (\Omega)$ (the same as in Theorem~\ref{thm:main}) such that 
\begin{align*}
 A_{\theta, \omega, 2} = U A_{\theta, \omega, 1} U^*
\end{align*}
holds.
\end{theorem}

Theorem~\ref{thm:main-j} is essentially a consequence of Theorem~\ref{thm:main} and the following proposition, which relates the resolvent of the Dirichlet operator $A_{\rm D}$
in \eqref{eq:AD}
to the resolvent of the operator $A_{\theta,\omega}$ via a perturbation term containg the Dirichlet-to-Neumann map and the function $\theta$.
We shall restrict elements in $H^{- 1/2} (\partial \Omega)$ to $\omega$ and use the operator
\begin{align}\label{eq:Pomega}
 P_\omega : H^{- 1/2} (\partial \Omega) \to \big\{ \psi |_\omega : \psi \in H^{- 1/2} (\partial \Omega) \big\}, \quad P_\omega \psi = \psi |_\omega;
\end{align}
here the restriction $\psi |_\omega$ is defined by $( \psi\vert_\omega,\phi):=( \psi,\phi)_{\partial\Omega}$ for all $\phi \in H^{1/2} (\partial \Omega)$ with $\supp \phi \subset \omega$.
One can view $P_\omega$ as the dual of the embedding operator from $H^{1/2}_\omega$ into $H^{1/2}(\partial\Omega)$.

\begin{proposition}\label{prop:Krein}
Let $\omega \subset \partial \Omega$ be an open, nonempty set, let $\theta\in L^\infty(\partial\Omega)$ be a complex-valued function such that 
$\theta\vert_{\partial\Omega\setminus\omega}=0$, and let
$A_{\theta, \omega}$ be the $m$-sectorial operator defined in~\eqref{eq:Athetaomega}. 
Then the operator $P_\omega (\theta + M (\lambda)) \!\upharpoonright\! H^{1/2}_\omega$ 
is injective for all $\lambda\in \rho (A_{\theta, \omega}) \cap \rho (A_{\rm D})$ and the identity
\begin{align}\label{eq:Krein}
 (A_{\theta, \omega} - \lambda)^{-1} = (A_{\rm D} - \lambda)^{-1} + \gamma (\lambda) 
 \bigl( P_\omega (\theta + M (\lambda)) \!\upharpoonright\! H^{1/2}_\omega \bigr)^{-1} P_\omega \gamma (\overline{\lambda})^* 
\end{align}
holds for all $\lambda \in \rho (A_{\theta, \omega}) \cap \rho (A_{\rm D})$.
\end{proposition}

\begin{proof}
We verify first that $P_\omega (\theta + M (\lambda)) \!\upharpoonright\! H^{1/2}_\omega$ is injective 
for $\lambda\in \rho (A_{\theta, \omega}) \cap \rho (A_{\rm D})$. 
Indeed, assume that $\psi \in H^{1/2}_\omega$ is such that $P_\omega (\theta + M (\lambda)) \psi = 0$, that is,
\begin{align*}
 \big( (\theta + M (\lambda) ) \psi, \phi \big)_{\partial \Omega} = 0, \quad \phi \in H^{1/2} (\partial \Omega),\,\, \supp \phi \subset \omega.
\end{align*}
Then $u_\lambda := \gamma (\lambda) \psi$ satisfies $\cL u_\lambda = \lambda u_\lambda$, $u_\lambda |_{\partial \Omega} \in H^{1/2}_\omega$, and 
\begin{align*}
 \big(\theta u_\lambda |_{\partial \Omega}+\partial_\cL u_\lambda |_{\partial \Omega}, \phi\big)_{\partial \Omega} = 0, 
 \quad \phi \in H^{1/2} (\partial \Omega),\,\, \supp \phi \subset \omega,
\end{align*}
which implies $u_\lambda \in \ker (A_{\theta, \omega} - \lambda)$ by \eqref{eq:Athetaomega}--\eqref{bcjussi}. 
Together with $\lambda \in \rho (A_{\theta, \omega})$ it follows $u_\lambda = 0$ and, thus, $\psi = u_\lambda |_{\partial \Omega} = 0$.

Let us now come to the proof of~\eqref{eq:Krein}. For this let $v \in L^2 (\Omega)$ be arbitrary. 
Since $\lambda \in \rho (A_{\theta, \omega}) \cap \rho (A_{\rm D})$, we can define
\begin{align}\label{eq:uw}
 u = (A_{\theta, \omega} - \lambda)^{-1} v - (A_{\rm D} - \lambda)^{-1} v \quad \text{and} \quad z = (A_{\theta, \omega} - \lambda)^{-1} v.
\end{align}
Then $u \in H^1 (\Omega)$ with $\cL u = \lambda u$, $z \in \dom A_{\theta, \omega}$, and $u |_{\partial \Omega} = z |_{\partial \Omega} \in H^{1/2}_\omega$. Moreover,
\begin{align*}
 \partial_\cL u |_{\partial \Omega} & = \partial_\cL z |_{\partial \Omega} - \partial_\cL \big( (A_{\rm D} - \lambda)^{-1} v \big) |_{\partial \Omega} 
 = \partial_\cL z |_{\partial \Omega} + \gamma (\overline{\lambda})^* v
\end{align*}
by Lemma~\ref{lem:DNproperties}~(i). For all $\psi \in H^{1/2} (\partial \Omega)$ with $\supp \psi \subset \omega$ we then obtain  
\begin{align*}
 \big( \gamma (\overline{\lambda})^* v, \psi \big)_{\partial \Omega} & 
 = ( \partial_\cL u |_{\partial \Omega} - \partial_\cL z |_{\partial \Omega}, \psi \big)_{\partial \Omega}  \\
 &= \big( M (\lambda) u |_{\partial \Omega} - \partial_\cL z |_{\partial \Omega}, \psi \big)_{\partial \Omega} = \big( (M (\lambda)  +  \theta) z |_{\partial \Omega}, \psi \big)_{\partial \Omega}.
\end{align*}
Hence $P_\omega \gamma (\overline{\lambda})^* v = P_\omega (\theta + M (\lambda)) z |_{\partial \Omega}$, that is, 
$P_\omega \gamma (\overline{\lambda})^* v \in \ran (P_\omega (\theta + M (\lambda))\! \upharpoonright\! H^{1/2}_\omega)$ and 
\begin{align*}
 \bigl( P_\omega (\theta + M (\lambda)) \!\upharpoonright\! H^{1/2}_\omega \bigr)^{-1} P_\omega \gamma (\overline{\lambda})^* v = 
 z |_{\partial \Omega} = u |_{\partial \Omega}.
\end{align*}
It follows
\begin{align*}
 \gamma (\lambda) \bigl( P_\omega (\theta + M (\lambda)) \!\upharpoonright \! H^{1/2}_\omega \bigr)^{-1} 
 P_\omega \gamma (\overline{\lambda})^* v = \gamma (\lambda) u |_{\partial \Omega} = u,
\end{align*}
which, together with the definition of $u$ in~\eqref{eq:uw}, completes the proof of~\eqref{eq:Krein}.
\end{proof}

\begin{proof}[Proof of Theorem~\ref{thm:main-j}]
Let $U$ be the unitary operator in $L^2 (\Omega)$ constructed in the proof of Theorem~\ref{thm:main}, which satisfies
\begin{align}\label{eq:jawollo}
 U \gamma_1 (\mu) \phi = \gamma_2 (\mu) \phi
\end{align}
for all $\mu \in \C \setminus \R$ and all $\phi \in H^{1/2} (\partial \Omega)$ with $\supp \phi \subset \omega$ as well as
\begin{equation}\label{holladi}
 U (A_{\rm D, 1} - \lambda)^{-1} = (A_{\rm D, 2} - \lambda)^{-1} U
\end{equation}
for $\lambda \in\rho (A_{\rm D, 1})\cap\rho (A_{\rm D, 2})$. Let us fix $\lambda \in (\C \setminus \R) \cap \rho (A_{\theta, \omega, 1}) \cap \rho (A_{\theta, \omega, 2})$.
Then with $P_\omega$ in~\eqref{eq:Pomega} the identity
\begin{align}\label{eq:gammaAdj}
 P_\omega \gamma_1 (\overline{\lambda})^* = P_\omega \gamma_2 (\overline{\lambda})^* U
\end{align}
holds. In fact, for $u \in L^2 (\Omega)$ and $\psi \in H^{1/2} (\partial \Omega)$ with $\supp \psi \subset \omega$ we have
\begin{align*}
 (\gamma_1 (\overline{\lambda})^* u, \psi)_{\partial\Omega} & = ( u, \gamma_1 (\overline{\lambda}) \psi)_{L^2(\Omega)} = 
 (u, U^* \gamma_2 (\overline{\lambda}) \psi)_{L^2(\Omega)} = (\gamma_2 (\overline{\lambda})^* U u, \psi)_{\partial\Omega}
\end{align*}
taking into account~\eqref{eq:jawollo}; this yields~\eqref{eq:gammaAdj}. Using Proposition~\ref{prop:Krein}, \eqref{holladi},
the assumption \eqref{eq:Mgleich22}, and~\eqref{eq:gammaAdj}, we obtain 
\begin{align*}
 U (A_{\theta, \omega, 1} - \lambda)^{-1} & = U (A_{\rm D, 1} - \lambda)^{-1} + 
  U \gamma_1 (\lambda) \bigl( P_\omega (\theta + M_1 (\lambda)) \!\upharpoonright\! H^{1/2}_\omega \bigr)^{-1} P_\omega \gamma_1 (\overline{\lambda})^* \\
 & = (A_{\rm D, 2} - \lambda)^{-1} U + \gamma_2 (\lambda) \bigl( P_\omega (\theta + M_2 (\lambda))\! \upharpoonright\! H^{1/2}_\omega \bigr)^{-1} P_\omega \gamma_2 (\overline{\lambda})^* U \\
 & = (A_{\theta, \omega, 2} - \lambda)^{-1} U.
\end{align*}
This yields $A_{\theta, \omega, 2} = U A_{\theta, \omega, 1} U^*$ and completes the proof.
\end{proof}

\section{An inverse problem for a self-adjoint Robin operator with partial Robin-to-Dirichlet data}\label{55555}

In this section we turn to an inverse problem for elliptic differential operators with Robin boundary conditions on the whole boundary of the unbounded
Lipschitz domain $\Omega$. 
In contrast to the previous section we restrict ourselves to self-adjoint boundary conditions. 
More specifically, for a real-valued function $\theta\in L^\infty(\partial \Omega)$ we consider the densely defined, semibounded, closed form
\begin{align*}
 \fa_\theta [u, v] = \fa [u, v] + ( \theta u |_{\partial \Omega}, v |_{\partial \Omega})_{\partial \Omega}, 
 \quad \dom \fa_\theta =  H^1 (\Omega) ,
\end{align*}
in $L^2(\Omega)$ and the corresponding semibounded, self-adjoint Robin operator
\begin{align*}
 A_\theta u = \cL u, \quad \dom A_\theta = \left\{ u \in H^1 (\Omega) : \cL u \in L^2 (\Omega), \partial_\cL u |_{\partial \Omega} + \theta u |_{\partial \Omega} = 0 \right\}.
\end{align*}
Our aim is to prove that this operator is determined uniquely up to unitary equivalence by the knowledge of a corresponding Robin-to-Dirichlet map 
on any nonempty, open subset of the boundary.

The following lemma prepares the definition of the Robin-to-Dirichlet map. It can be proved analogously to Lemma~\ref{lem:BVPwelldef}. 

\begin{lemma}\label{lem:RobinBVP}
For each $\lambda \in \rho (A_\theta)$ and each $\psi \in H^{- 1/2} (\partial \Omega)$ the boundary value problem
\begin{align*}
 \cL u = \lambda u, \qquad \partial_\cL u |_{\partial \Omega} + \theta u |_{\partial \Omega} = \psi,
\end{align*}
has a unique solution $u_\lambda \in H^1 (\Omega)$. 
\end{lemma}

Due to Lemma~\ref{lem:RobinBVP} the following definition makes sense.

\begin{definition}
For $\lambda \in \rho (A_\theta)$ the {\em Robin-to-Dirichlet map} $M_\theta (\lambda)$ is defined by 
\begin{align*}
M_\theta (\lambda) : H^{- 1/2} (\partial \Omega) \to H^{1/2} (\partial \Omega),\quad 
 M_\theta (\lambda) \big( \partial_\cL u_\lambda |_{\partial \Omega} + \theta u_\lambda |_{\partial \Omega} \big) := u_\lambda |_{\partial \Omega},
\end{align*}
for each $u_\lambda \in H^1 (\Omega)$ satisfying $\cL u_\lambda = \lambda u_\lambda$.
\end{definition}

For $\lambda\in \rho (A_\theta)$ we also define the {\em Poisson operator for the Robin problem} $\gamma_\theta (\lambda)$ by
\begin{align}\label{eq:PoissonRobin}
 \gamma_\theta (\lambda) : H^{- 1/2} (\partial \Omega) \to L^2 (\Omega),\quad 
 \gamma_\theta (\lambda) \big( \partial_\cL u_\lambda |_{\partial \Omega} + \theta u_\lambda |_{\partial \Omega} \big) := u_\lambda,
\end{align}
for any $u_\lambda \in H^1 (\Omega)$ such that $\cL u_\lambda = \lambda u_\lambda$.

In order to prove the main result of this section we collect some properties of $\gamma_\theta (\lambda)$ and $M_\theta (\lambda)$, which are analogs 
of the statements in Lemma~\ref{lem:DNproperties}. Their proofs are similar to those in \cite[Lemma 2.4]{BR12} and are not repeated here.

\begin{lemma}\label{lem:RDproperties}
For $\lambda, \mu \in \rho (A_\theta)$ let $\gamma_\theta (\lambda), \gamma_\theta (\mu)$ be the Poisson operators for the Robin problem and let $M_\theta (\lambda), M_\theta (\mu)$ be the Robin-to-Dirichlet maps. Then the following assertions hold.
\begin{enumerate}
 \item $\gamma_\theta (\lambda)$ is bounded and 
 the identity
 \begin{align*}
  \gamma_\theta (\lambda) = \big( I + (\lambda - \mu) (A_\theta - \lambda)^{-1} \big) \gamma_\theta (\mu)
 \end{align*}
 holds.
 \item $M_\theta (\lambda)$ is a bounded operator from $H^{-1/2}(\partial\Omega)$ to $H^{1/2}(\partial\Omega)$, the operator function $\lambda \mapsto M_\theta (\lambda)$ 
 is holomorphic on $\rho (A_\theta)$, and
 \begin{align*}
  (\Imag \mu)\|\gamma_\theta (\mu) \phi\|_{L^2(\Omega)}^2 = \Imag ( M_\theta (\mu) \phi, \phi)_{\partial\Omega}
 \end{align*}
 holds for all $\phi \in H^{- 1/2} (\partial \Omega)$.
\end{enumerate}
\end{lemma}

For $\phi\in H^{-1/2}(\partial\Omega)$ and an open set $\nu\subset\partial\Omega$ we shall say that $\phi$ vanishes on $\nu$ if
$(\phi,\eta)_{\partial\Omega}=0$ for all $\eta\in H^{1/2}(\partial\Omega)$ with $\supp\eta\subset\nu$. As usual, 
we define the support $\supp\phi \subset \partial \Omega$ of $\phi$ to be the complement of the union of all open sets 
on which $\phi$ vanishes.

% \begin{equation*}
%  \supp\phi=\partial\Omega\setminus\bigcup\, \bigl\{\nu\subset\partial\Omega\,\,\text{open}\!:\varphi\,\,\text{vanishes on}\,\,\nu\bigr\}.
% \end{equation*}

The main result of this section is the following.

\begin{theorem}\label{thm:main2}
Let $\cL_1, \cL_2$ be two uniformly elliptic differential expressions on $\Omega$ of the form~\eqref{eq:diffexpr} with 
coefficients $a_{jk, 1}, a_{j, 1}, a_1$ and $a_{jk, 2}, a_{j, 2}, a_2$, respectively, satisfying \eqref{eq:ajkaj}--\eqref{eq:a}. Let $\theta_1,\theta_2
\in L^\infty(\partial\Omega)$ 
be real-valued and let $A_{\theta_1}$, $A_{\theta_2}$ and $M_{\theta_1} (\lambda), M_{\theta_2} (\lambda)$ 
denote the corresponding self-adjoint Robin operators and Robin-to-Dirichlet maps, respectively. Assume that $\omega \subset \partial \Omega$ is an open, nonempty set such that
\begin{align}\label{eq:Mthetagleich}
 ( M_{\theta_1} (\lambda) \phi, \phi)_{\partial \Omega} = ( M_{\theta_2} (\lambda) \phi, \phi)_{\partial \Omega}, \quad 
 \phi\in H^{- 1/2} (\partial \Omega),\, \supp \phi \subset \omega,
\end{align}
holds for all $\lambda \in \cD$, where $\cD \subset \rho (A_{\theta_1}) \cap \rho (A_{\theta_2})$ is a set with an accumulation point in 
$\rho (A_{\theta_1}) \cap \rho (A_{\theta_2})$. Then there exists a unitary operator $U$ in $L^2 (\Omega)$ such that 
\begin{align*}
 A_{\theta_2} = U A_{\theta_1} U^*
\end{align*}
holds.
\end{theorem}

\begin{proof}
The proof of Theorem~\ref{thm:main2} is a modification of the proof of Theorem~\ref{thm:main} and we will leave some details to the reader. 
For any $\mu \in \C\setminus\R$ and let $\gamma_{\theta_i} (\mu)$ be the Poisson operator for the Robin problem as 
defined in~\eqref{eq:PoissonRobin}, $i = 1, 2$. We define a linear mapping $V$ in $L^2 (\Omega)$ on the domain
\begin{align*}
 \dom V=\spann \big\{\gamma_{\theta_1} (\mu)\phi : \phi\in H^{- 1/2} (\partial \Omega), \,\supp \phi \subset \omega,\,\mu \in \C \setminus \R \big\}
\end{align*}
setting
\begin{align*}
 V \gamma_{\theta_1} (\mu) \phi = \gamma_{\theta_2} (\mu) \phi, \quad \phi\in H^{- 1/2} (\partial \Omega), \,\supp \phi \subset \omega,\,\mu \in \C \setminus \R,
\end{align*}
and extending this operator by linearity to all of $\dom V$. Clearly, we have
\begin{align*}
 \ran V = \spann \big\{\gamma_{\theta_2} (\mu)\phi : \phi\in H^{- 1/2} (\partial \Omega),\, \supp \phi \subset \omega,\,\mu \in \C \setminus \R \big\}.
\end{align*}
As in Step~1 of the proof of Theorem~\ref{thm:main} we conclude from~\eqref{eq:Mthetagleich} with the help of Lemma~\ref{lem:RDproperties}~(i) and~(ii) 
(instead of Lemma~\ref{lem:DNproperties}~(ii) and~(iii)) that $V$ is well-defined, isometric, and satisfies 
\begin{align}\label{eq:Vidtheta}
 V (A_{\theta_1} - \lambda)^{-1} \upharpoonright H_\lambda = (A_{\theta_2} - \lambda)^{-1} V \upharpoonright H_\lambda
\end{align}
for each fixed $\lambda \in \C \setminus \R$, where $H_\lambda$ is the subspace of $\dom V$ given by
\begin{align*}
 H_\lambda = \spann \big\{\gamma_{\theta_1} (\mu) \phi : \phi\in H^{1/2}(\partial\Omega),\,\supp\phi\subset\omega,\,\mu \in \C \setminus \R, \mu \neq \lambda \big\}.
\end{align*}

Let us now check that $\dom V$ is dense in $L^2 (\Omega)$. Let $\widetilde \Omega$ and $\widetilde \cL_1$ be defined as in Step~2 of the proof of Theorem~\ref{thm:main} above with the additional condition that there exist  $\omega_0\subset\partial\Omega$ such that $\overline \omega_0\subset \omega$ and still $\partial\Omega\setminus\omega_0\subset\partial\widetilde\Omega$. Define the real-valued function $\widetilde \theta_1 \in L^\infty(\partial \widetilde \Omega)$ by
\begin{align*}
 \widetilde \theta_1 = \begin{cases}
                      \theta_1 & \text{on}~\partial \Omega \setminus \omega_0, \\
		      0 & \text{otherwise}.
                     \end{cases}
\end{align*}
Then  the operator
\begin{align*}
 \widetilde A_{\widetilde \theta_1} \widetilde u = \widetilde \cL_1 \widetilde u, \quad \dom \widetilde A_{\widetilde \theta_1} 
 = \bigl\{ \widetilde u \in H^1 (\widetilde \Omega) : \widetilde \cL_1 \widetilde u \in L^2 (\widetilde \Omega),~\partial_{\widetilde \cL_1} \widetilde u \big|_{\partial \widetilde \Omega} 
 + \widetilde \theta_1 \widetilde u |_{\partial \widetilde \Omega} = 0 \bigr\},
\end{align*}
in $L^2 (\widetilde \Omega)$ is self-adjoint and semibounded from below; as in the proof of Theorem~\ref{thm:main} one argues that $\widetilde A_{\widetilde \theta_1}$
can be assumed to be uniformly positive. For each $\widetilde v \in L^2 (\widetilde \Omega)$ such that 
$\widetilde v$ vanishes on $\Omega$, we define
\begin{align*}
 \widetilde u_{\mu, \widetilde v} = (\widetilde A_{\widetilde \theta_1} - \mu)^{-1} \widetilde v, \quad \mu \in \C \setminus \R. 
\end{align*}
Moreover, we denote by $u_{\mu, \widetilde v}$ the restriction of $\widetilde u_{\mu, \widetilde v}$ to $\Omega$. 
Then $\cL_1 u_{\mu, \widetilde v} = \mu u_{\mu, \widetilde v}$ and 
by construction 
\begin{equation}\label{istdasklar}
 \supp \bigl(\partial_{\cL_1} u_{\mu, \widetilde v}|_{\partial \Omega} + \theta_1 u_{\mu, \widetilde v} |_{\partial \Omega}\bigr) \subset \overline{\omega_0}\subset \omega.
\end{equation}
In fact, to justify \eqref{istdasklar} consider $x \in \partial \Omega \setminus \overline{\omega_0}$, choose 
an open set $\nu \subset \partial \Omega \setminus \overline{\omega_0}$ with $x \in \nu$, and let $\phi \in H^{1/2} (\partial \Omega)$ with 
$\supp \phi \subset \nu$. The first inclusion in \eqref{istdasklar} follows if we show
\begin{equation}\label{la2}
( \partial_{\cL_1} u_{\mu, \widetilde v} |_{\partial \Omega} + \theta_1 u_{\mu, \widetilde v} |_{\partial \Omega}, \phi )_{\partial \Omega}=0.
\end{equation}
Choose $w \in H^1 (\Omega)$ with $w |_{\partial \Omega} = \phi$ so that, in particular, $w |_{\omega_0} = 0$. 
Hence the extension $\widetilde w$ by zero of $w$ onto $\widetilde\Omega$ satisfies $\widetilde w \in H^1 (\widetilde \Omega)$ and 
$\supp (\widetilde w |_{\partial \widetilde \Omega}) \subset \partial \Omega \setminus \overline{\omega_0}$. Now it follows from the definition of the conormal derivative that
\begin{equation*}
\begin{split}
 &( \partial_{\cL_1} u_{\mu, \widetilde v} |_{\partial \Omega} + \theta_1 u_{\mu, \widetilde v} |_{\partial \Omega}, \phi )_{\partial \Omega} \\
 &\qquad\qquad = - (\cL_1 u_{\mu, \widetilde v}, w)_{L^2 (\Omega)} + \fa [u_{\mu, \widetilde v}, w] + ( \theta_1 u_{\mu, \widetilde v} |_{\partial \Omega}, w |_{\partial \Omega} )_{\partial \Omega} \\
 &\qquad\qquad = - (\widetilde \cL_1 \widetilde u_{\mu, \widetilde v}, \widetilde w)_{L^2 (\widetilde \Omega)} + 
                  \widetilde \fa [\widetilde u_{\mu, \widetilde v}, \widetilde w] + 
                  ( \widetilde \theta_1 \widetilde u_{\mu, \widetilde v} |_{\partial \widetilde\Omega}, \widetilde w |_{\partial \widetilde\Omega} )_{\partial \widetilde\Omega} \\                  
 &\qquad\qquad = ( \partial_{\widetilde \cL_1} \widetilde u_{\mu, \widetilde v}|_{\partial \widetilde\Omega} + \widetilde \theta_1 \widetilde u_{\mu, \widetilde v} |_{\partial \widetilde\Omega}, 
 \widetilde w |_{\partial \widetilde\Omega} )_{\partial \widetilde\Omega}
 = 0,
\end{split}
 \end{equation*}
which proves \eqref{la2} and therefore \eqref{istdasklar} holds.
Now it follows in the same way as in the proof of Theorem~\ref{thm:main} that
$u_{\mu, \widetilde v} \in \dom V$ for all $\mu \in \C \setminus \R$ and all $\widetilde v \in L^2 (\widetilde \Omega)$ with $\widetilde v |_\Omega = 0$.

If we choose $u \in L^2 (\Omega)$ being orthogonal to $\dom V$ and denote by $\widetilde u$ the extension of $u$ by zero to $\widetilde \Omega$ then we obtain
\begin{align*}
 0 = (u, u_{\overline \mu, \widetilde v})_{L^2 (\Omega)} = \big( \widetilde u, (\widetilde A_{\widetilde \theta_1} - \overline \mu)^{-1} 
 \widetilde v \big)_{L^2 (\widetilde \Omega)} = \big( (\widetilde A_{\widetilde \theta_1} - \mu)^{-1} \widetilde u, \widetilde v \big)_{L^2 (\widetilde \Omega)}
\end{align*}
for all $\mu \in \C \setminus \R$ and all $\widetilde v \in L^2 (\widetilde \Omega)$ which vanish on $\Omega$, 
that is, $$\big((\widetilde A_{\widetilde \theta_1} - \mu)^{-1} \widetilde u \big) |_{\widetilde \Omega \setminus \Omega} = 0$$ for all $\mu \in \C \setminus \R$. 
Proceeding further as in Step~2 of the proof of Theorem~\ref{thm:main} it can be concluded that $e^{- t \sqrt{\widetilde A_{\widetilde \theta_1}}} \widetilde u$ 
vanishes on an open, nonempty subset of $\widetilde \Omega \times (0, \infty)$ and by unique continuation it follows 
$e^{- t \sqrt{\widetilde A_{\widetilde \theta_1}}} \widetilde u = 0$ on $\widetilde \Omega$ for each $t > 0$. Hence, $u = 0$, 
which implies that $\dom V$ is dense in $L^2 (\Omega)$. Analogously one shows that $\ran V$ is dense in $L^2 (\Omega)$. 

Now it follows in the same way as in the end of Step 2 of the proof of Theorem~\ref{thm:main} that
the isometric operator $V$ extends by continuity to a unitary operator $U : L^2 (\Omega) \to L^2 (\Omega)$ and that~\eqref{eq:Vidtheta} extends to
\begin{align*}
 U (A_{\theta_1} - \lambda)^{-1} = (A_{\theta_2} - \lambda)^{-1} U.
\end{align*}
This yields $A_{\theta_2} = U A_{\theta_1} U^*$ and hence completes the proof of the theorem.
\end{proof}

\begin{ack}
J.R.\ gratefully acknowledges financial support by the grant no.\ 2018-04560 of the Swedish Research Council (VR).
\end{ack}

\end{document}